\documentclass[leqno,CJK]{siamltex704}
\usepackage{amssymb,amsmath,graphicx,amscd,mathrsfs}
\usepackage{color,xcolor,amsmath}
\usepackage{longtable}
\usepackage{amsmath}
\usepackage{graphicx}
\usepackage{mathrsfs}
\usepackage{float}
\usepackage{amsfonts,amssymb}
\usepackage{dsfont}
\usepackage{pifont}
\usepackage{hyperref}
\usepackage{multirow}
\usepackage{placeins}
\numberwithin{equation}{section}
\def\3bar{{|\hspace{-.02in}|\hspace{-.02in}|}}
\def\E{{\mathcal{E}}}
\def\T{{\mathcal{T}}}

\def\dQ{{\mathbb{Q}}}

\def\b0{\boldsymbol{0}}
\def\sumT{\sum_{T\in\mathcal{T}_h}}     

\def\bn{{\mathbf{n}}}

\def\bf{{\mathbf{f}}}
\def\bq{{\mathbf{q}}}

\newtheorem{remark}{Remark}[section]
\newtheorem{algorithm1}{Weak Galerkin Algorithm}

 \newcommand{\eps}{\varepsilon}

 \newcommand{\Real}{\mathbb{R}}

 \newcommand{\trb}[1]{|\!|\!|#1|\!|\!|}

\allowdisplaybreaks 

\begin{document}

\title{The weak Galerkin method for eigenvalue problems}

\author{Hehu Xie\thanks{LSEC and Institute of Computational Mathematics and Scientific/Engineering Computing,
Academy of Mathematics and Systems Science, Chinese Academy of
Sciences, Beijing 100190, China (hhxie@lsec.cc.ac.cn) } \and Qilong
Zhai\thanks{Department of Mathematics, Jilin University, Changchun,
China (diql15@mails.jlu.edu.cn). } \and Ran Zhang\thanks{Department
of Mathematics, Jilin University, Changchun, China
(zhangran@mail.jlu.edu.cn). The research of Zhang was supported in
part by China Natural National Science Foundation(11271157,
11371171, 11471141), and by the Program for New Century Excellent
Talents in University of Ministry of Education of China.}}

\maketitle

\begin{abstract} This article is devoted to computing the eigenvalue
of the Laplace eigenvalue problem by the weak Galerkin (WG) finite
element method with emphasis on obtaining lower bounds. The WG
method is on the use of weak functions and their weak derivatives
defined as distributions. Weak functions and weak derivatives can be
approximated by polynomials with various degrees. Different
combination of polynomial spaces leads to different WG finite
element methods, which makes WG methods highly flexible and
efficient in practical computation.  We establish the optimal-order
error estimates for the WG finite element approximation for the
eigenvalue problem. Comparing with the classical
nonconforming finite element method which can just provide lower
bound approximation by linear elements with only the second order
convergence, the WG methods can naturally provide lower bound
approximation with a high order convergence (larger than $2$). Some
numerical results are also presented to demonstrate the efficiency
of our theoretical results.
\end{abstract}



\section{Introduction}

The study of eigenvalues and eigenfunctions of partial differential
operators both in theoretical and approximation grounds is very
important in many fields of sciences, such as quantum mechanics,
fluid mechanics, stochastic process, structural mechanics, etc.
Thus, a fundamental work is to find the eigenvalues and
corresponding eigenfunctions of partial differential operators.

In this paper, we consider the following model problem: Find
$(\lambda, u)$ such that
\begin{equation}\label{problem-eq}
\left\{
\begin{array}{rcl}
-\Delta u &=& \lambda u,\quad \text{in }\Omega,\\
u &=& 0,\quad~~ \text{on }\partial\Omega,\\
\int_\Omega u^2d\Omega &=& 1,
\end{array}
\right.
\end{equation}
where $\Omega$ is a polyhedral domain in $\Real^d$ $(d=2,3)$. For
simplicity, we are only concerned with the case $d=2$, while all the
conclusions can be extended to $d=3$ trivially. The classical
variational form of problem (\ref{problem-eq}) is defined as
follows: Find $u\in H_0^1(\Omega)$ and $\lambda\in\Real$ such that
$b(u,u)=1$ and
\begin{eqnarray}\label{vartion-form}
a(u,v)=\lambda b(u,v),\quad \forall v\in H_0^1(\Omega),
\end{eqnarray}
where
\begin{eqnarray*}
a(u,v)=(\nabla u,\nabla v)\ \ \ \ \ {\rm and}\ \ \ \ \ b(u,v)=(u,v).
\end{eqnarray*}
It is well known that the problem (\ref{vartion-form}) has the eigenvalue sequence \cite{IBEig}
\begin{eqnarray*}
0<\lambda_1\leq \lambda_2\leq\ldots\leq\lambda_j\leq\ldots\longrightarrow +\infty
\end{eqnarray*}
with the corresponding eigenfunction sequence
\begin{eqnarray*}
u_1,u_2,\ldots,u_j,\ldots,
\end{eqnarray*}
which satisfies the property $b(u_i,u_j)=\delta_{ij}$
($\delta_{ij}$ denotes the Kronecker function).

The first aim of this paper is to analyze the weak Galerkin finite element method
for the eigenvalue problem. We will give the corresponding error estimates
 for the eigenpair approximations by the
weak Galerkin finite element method based on the standard theory
from \cite{IBEig,Boffi10}.

The eigenvalue is a positive real number, and thus it is credible if
we get both the upper and lower bounds. In fact, a simple
combination of lower and upper bounds will present intervals to
which exact eigenvalue belongs. For the Laplace eigenvalue problems,
since the Rayleigh quotient and minimum-maximum principle, it is
easily to obtain the upper bounds of eigenvalue by any standard
conforming finite element methods \cite{Feng65, SF73}. For the lower
bounds, the computation is of high interest and generally more
difficult. Influenced by the minimum-maximum principle, people try
to obtain the lower bounds with the nonconforming element methods.
In fact, the lower bound property of eigenvalues by nonconforming
elements are observed in numerical aspects at the beginning
\cite{Boffi10, LTZ0505, RT92, SW10, WTDG73, ZC67}. After that, a
series of results make progress in this aspect, e.g., Lin and Lin
\cite{LL06} proved that the non-conforming ${\rm EQ}_1^{\rm rot}$
rectangular element approximates exact eigenvalues associated with
smooth eigenfunctions from below in 2006.  Hu, Huang and Shen
\cite{HHS04} gets the lower bound of Laplace eigenvalue problem by
conforming linear and bilinear elements together with the mass
lumping method.
A general kind of expansion method, which was first proved in its
full term in \cite{ZhangYangChen} by the similar argument in
\cite{AD}, is extensively used in \cite{HuHuangLin, Li08,LuoLinXie}
and the references cited therein. Another interesting way is
provided in \cite{Liu} and the corresponding paper cited therein.

Our work was inspired by some recent studies of lower end
approximation of eigenvalues by finite element discretization for
some elliptic partial differential operators
 \cite{AD,HuHuangLin,LuoLinXie, LX1,ZhangYangChen}. One
crucial technical ingredient that is needed in the analysis
is some lower bound of the eigenfunction
discretization error by the finite element method. The mainly
challenge aspect is to design the high order elements to present the
lower bound scheme. So far, the lower bound finite element methods are
mainly first order nonconforming finite element methods. It is desired
to design some high order numerical methods to obtain the lower bound
eigenvalue approximations which is the second aim of this paper.

Recently a new class of finite element methods, called weak Galerkin
(WG) finite element methods have been developed for the partial
differential equations for its highly flexible and robust
properties. The WG method refers to a numerical scheme for partial
differential equations in which differential operators are
approximated by weak forms as distributions over a set of
generalized functions. It has been demonstrated that the WG method
is highly flexible and robust as a numerical technique that employs
discontinuous piecewise polynomials on polygonal or polyhedral
finite element partitions. This thought was first proposed in
\cite{WY13} for a model second order elliptic problem in 2012, and
further developed in \cite{LW13, mwy, mwy0927, mwwy0618, mwy3818,
mwyz-maxwell, mwyz-biharmonic, WY14, ZZ15} with other applications.
In order to enforce necessary weak continuities for approximating
functions, proper stabilizations are employed. The main advantages
of the WG method include the finite element partition can be of
polytopal type with certain regular requirements, the weak finite
element space is easy to construct with any given approximation
requirement and the WG schemes can be hybridized so that some
unknowns associated with the interior of each element can be locally
eliminated, yielding a system of linear equations involving much
less number of unknowns than what it appears.

The objective of the present paper is twofold. First, we will
introduce weak Galerkin method for solving the Laplacian eigenvalue
problem, which has the optimal-order error estimates. Furthermore we
investigate the performance of the WG methods for presenting the
guaranteed lower bound approximation of the eigenvalues problem
under the assumption that the global mesh-size is sufficiently
small. To demonstrate the potential of WG finite element methods in
solving eigenvalue problems, we will restrict ourselves to any order
WG elements (even for higher order WG finite element spaces) and
investigate the robustness and effectiveness of this method.

An outline of the paper is as follows. In Section 2, the necessary
notations, definitions of weak functions and weak derivatives are
introduced. The WG finite element scheme of the Laplace eigenvalue
problem is stated. Error estimates for the boundary value problem
are presented in Section 3. In Section 4, we establish the error
estimates for the WG finite element approximation for the eigenvalue
problem. Section 5 is devoted to presenting any higher order accuracy
lower bound approximation of the eigenvalues. Some numerical results
are presented in Section 6 to demonstrate the efficiency of our
theoretical results and some concluding remarks are given in the
last section.

\section{The weak Galerkin scheme for the eigenvalue problem}

\subsection{Preliminaries and notations}

First, we present some notation which will be used in this paper. We denote $(\cdot, \cdot)_{m,\omega}$
and $\|\cdot\|_{m,\omega}$ the inner-product and the norm on $H^m(\omega)$. If the region
$\omega$ is an edge or boundary of some element, we use $\langle\cdot,\cdot\rangle
_{m,\omega}$ instead of $(\cdot, \cdot)_{m,\omega}$. We shall drop the subscript
when $m=0$ or $\omega=\Omega$. In this paper, $P_r(\omega)$ denotes the space of polynomials on $\omega$
with degree no more than $r$. Throughout this paper, $C$ denotes a generic positive constant which
is independent of the mesh size.

Let $\T_h$ be a partition of the domain $\Omega$, and the elements
in $\T_h$ are polygons satisfying the regular assumptions specified
in \cite{WY14}. Denote by $\E_h$ the edges in $\T_h$, and by
$\E_h^0$ the interior edges $\E_h\backslash \partial\Omega$. For
each element $T\in\T_h$, $h_T$ represents the diameter of $T$, and
$h=\max_{T\in\T_h} h_T$ denotes the mesh size.

\subsection{A weak Galerkin scheme}
Now we introduce a weak Galerkin scheme solving the problem (\ref{problem-eq}).
 For a given integer $k\ge 1$, define the Weak Galerkin (WG) finite
element space
\begin{eqnarray*}
V_h=\{v=(v_0,v_b):v_0|_T\in P_k(T), v_b|_e\in P_{k-1}(e), \forall
T\in\T_h, e\in\E_h, \text{ and }v_b=0 \text{ on }\partial\Omega\}.
\end{eqnarray*}
For each weak function $v\in V_h$, we can define its weak gradient
$\nabla_w v$ by distribution element-wisely as follows.
\begin{definition}
For each $v\in V_h$, $\nabla_w v|_T$ is the unique polynomial in
$[P_{k-1}(T)]^2$ satisfying
\begin{eqnarray}\label{def-wgradient}
(\nabla_w v,\bq)_T=-(v_0,\nabla\cdot\bq)_T+\langle v_b,\bq\cdot\bn
\rangle_{\partial T},\quad\forall\bq\in [P_{k-1}(T)]^2,
\end{eqnarray}
where $\bn$ denotes the outward unit normal vector.
\end{definition}

For the aim of analysis, some projection operators are also employed in this paper. Let $Q_0$
denote the $L^2$ projection from $L^2(T)$ onto $P_k(T)$, $Q_b$
denote the $L^2$ projection from $L^2(e)$ onto $P_{k-1}(e)$, and $\dQ_h$
denote the $L^2$ projection from $[L^2(T)]^2$ onto $[P_{k-1}(T)]^2$.
Combining $Q_0$ and $Q_b$ together, we can define $Q_h=\{Q_0,Q_b\}$,
which is a projection from $H_0^1(\Omega)$ onto $V_h$.

Now we define three bilinear forms on $V_h$ that for any $v,w\in V_h$,
\begin{eqnarray*}
s(v,w)&=&\sumT h_T^{-1+\eps}\langle Q_b v_0-v_b,Q_b w_0-w_b\rangle_{\partial T},
\\
a_w(v,w)&=&(\nabla_w v,\nabla_w w)+s(v,w),
\\
b_w(v,w)&=&(v_0,w_0),
\end{eqnarray*}
where $0\le\eps <1$ is a small constant parameter to be selected.
With these preparations we can give the following weak Galerkin algorithm.

\begin{algorithm1}
Find $u_h\in V_h$, $\lambda_h\in\Real$ such that $b_w(u_h,u_h)=1$ and
\begin{eqnarray}\label{WG-scheme}
a_w(u_h,v)=\lambda_h b_w(u_h,v),\quad\forall v\in V_h.
\end{eqnarray}
\end{algorithm1}

\section{Error estimates for the boundary value problem}

In order to analyze the error of the eigenvalue problem by the weak Galerkin
method, we need some estimates for the boundary value problem. The main idea is similar to
\cite{WY2} but some modifications.

\subsection{A weak Galerkin method for the Poisson equation}
In this section, we consider the weak Galerkin method for
the following Poisson equation
\begin{equation}\label{Poisson-eq1}
\left\{
\begin{array}{rcl}
-\Delta u&=&f,\quad\text{in }\Omega,\\
u &=& 0,\quad \text{on }\partial\Omega.
\end{array}
\right.
\end{equation}
The corresponding weak Galerkin scheme is to find $u_h\in V_h$ such
that
\begin{eqnarray}\label{WG-scheme1}
a_w(u_h,v)=(f,v_0),\quad \forall v\in V_h.
\end{eqnarray}

Define a semi-norm on $V_h$ as follows
\begin{eqnarray*}
\trb{v}^2=a_w(v,v),\quad\forall v\in V_h.
\end{eqnarray*}
We claim that $\trb{\cdot}$ is indeed a norm on $V_h$. In order to check the
positive property, suppose
$\trb v=0$ . Then we have $\nabla_w v=0$ in $T$ and $Q_b v_0=v_b$ on $\partial T$
for all $T\in\T_h$. It follows that
\begin{eqnarray*}
(\nabla v_0,\nabla v_0)_T&=&-(v_0,\nabla\cdot\nabla v_0)_T+\langle v_0,
\nabla v_0\cdot\bn\rangle_{\partial T}
\\
&=&-(v_0,\nabla\cdot\nabla v_0)_T+\langle v_b,\nabla v_0\cdot\bn\rangle_{\partial T}
+\langle Q_b v_0-v_b,\nabla v_0\cdot\bn\rangle_{\partial T}
\\
&=&(\nabla v_0,\nabla_w v)_T=0,
\end{eqnarray*}
so that $v_0$ is piecewise constant and $v_b=Q_b v_0=v_0$ on $\partial T$.
Notice that $v_b=0$ on $\partial\Omega$, we can obtain that $v=0$.
For the analysis, we also define another norm on $V_h$ as
\begin{eqnarray*}
\trb v_1^2=\|\nabla_w v\|^2+\sumT h_T^{-1}\|Q_b v_0-v_b\|^2_{\partial T}.
\end{eqnarray*}
Furthermore, it is easy to check that the weak Galerkin scheme (\ref{WG-scheme1})
is symmetric and positive definite, which has a unique solution.

The following commutative property plays an essential role in the
forthcoming proof, which shows that the weak gradient operator is an
approximation of the classical gradient operator.
\begin{lemma}\rm{(\cite{WY2})}
For any element $T\in\T_h$, the following commutative property holds true,
\begin{eqnarray}\label{commu-prop}
\nabla_w(Q_h\varphi)=\dQ_h(\nabla\varphi),\quad\forall\varphi\in H^1(T).
\end{eqnarray}
\end{lemma}
\begin{proof}
From the definition of the weak gradient (\ref{def-wgradient}) and the
integration by parts, we have that for any $\bq\in [P_{k-1}(T)]^2$
\begin{eqnarray*}
(\nabla_w(Q_h\varphi),\bq)_T&=&-(Q_0\varphi,\nabla\cdot\bq)_T
+\langle Q_b\varphi,\bq\cdot\bn\rangle_{\partial T}
\\
&=&-(\varphi,\nabla\cdot\bq)_T
+\langle \varphi,\bq\cdot\bn\rangle_{\partial T}
\\
&=&(\nabla\varphi,\bq)_T=(\dQ_h(\nabla\varphi),\bq)_T,
\end{eqnarray*}
which completes the proof.
\end{proof}

\subsection{Error equation}
Suppose $u$ is the solution of (\ref{Poisson-eq1}), and
$u_h$ is the numerical solution of (\ref{WG-scheme1}). Denote by $e_h$
the error that
\begin{eqnarray*}
e_h=Q_h u-u_h=\{Q_0 u-u_0,Q_b u-u_b\}.
\end{eqnarray*}
Then $e_h$ should satisfy the following equation.
\begin{lemma}
Let $e_h$ be the error of the weak Galerkin scheme (\ref{WG-scheme1}). Then, for any $v\in V_h$,
we have
\begin{eqnarray}\label{error-eqn}
a_w(e_h,v)=\ell(u,v)+s(Q_h u,v),
\end{eqnarray}
where
\begin{eqnarray*}
\ell(u,v)=\sumT\langle(\nabla u-\dQ_h\nabla u)\cdot\bn, v_0-v_b\rangle_{\partial T}.
\end{eqnarray*}
\end{lemma}
\begin{proof}
From the definition of the weak gradient (\ref{def-wgradient}) and the commutative
property (\ref{commu-prop}), we can obtain on each element $T\in\T_h$ that
\begin{eqnarray*}
(\nabla_w Q_h u,\nabla_w v)_T&=&(\dQ_h\nabla u,\nabla_w v)_T
\\
&=&-(v_0,\nabla\cdot\dQ_h\nabla u)_T+\langle v_b,\dQ_h(\nabla u)\cdot\bn\rangle
_{\partial T}
\\
&=&(\nabla v_0,\dQ_h \nabla u)_T-\langle v_0-v_b,\dQ_h(\nabla u)\cdot\bn\rangle
_{\partial T}
\\
&=&(\nabla v_0,\nabla u)_T-\langle\dQ_h(\nabla u)\cdot\bn, v_0-v_b\rangle
_{\partial T}.
\end{eqnarray*}
Summing over all elements and it follows that
\begin{eqnarray*}
(\nabla_w Q_h u,\nabla_w v)&=&(\nabla v_0,\nabla u)-\sumT\langle\dQ_h(\nabla u)
\cdot\bn, v_0-v_b\rangle_{\partial T}
\\
&=&(f,v_0)+\sumT\langle\nabla u\cdot\bn,v_0\rangle_{\partial T}
-\sumT\langle\dQ_h(\nabla u)\cdot\bn, v_0-v_b\rangle_{\partial T}
\\
&=&(f,v_0)+\ell(u,v).
\end{eqnarray*}
Notice that the numerical solution $u_h$ satisfies (\ref{WG-scheme1}).
Then we can derive that
\begin{eqnarray*}
a_w(e_h,v)=\ell(u,v)+s(Q_h u,v),\ \ \ \ \forall v\in V_h,
\end{eqnarray*}
which completes the proof.
\end{proof}

In order to estimate the right hand side terms of (\ref{error-eqn}),
we still need some technique tools introduced in \cite{WY14}.

\begin{lemma}{\rm(\cite{WY14})}\label{Trace inequality}
~\emph{\rm (}Trace Inequality{\rm)} Let $\mathcal{T}_h$ be a
partition of the domain $\Omega$ into polygons in 2D or polyhedra in
3D. Assume that the partition $\mathcal{T}_h$ satisfies the
Assumptions A1, A2, and A3 as stated in \cite{WY14}. Let $p>1$ be
any real number. Then, there exists a constant $C$ such that for any
$T\in \mathcal{T}_h$ and edge/face $e\in\partial T$, we have
\begin{eqnarray}\label{Trace inequality00}
\|\theta\|^p_{L^p(e)}\leq
Ch_T^{-1}\big(\|\theta\|^p_{L^p(T)}+h^p_T\|\nabla\theta\|^p_{L^p(T)}\big),
\end{eqnarray}
for any  $\theta\in W^{1,p}(T)$. 
\end{lemma}

\begin{lemma}{\rm(\cite{WY14})}\label{Inverse Inequality}~{\rm (}Inverse Inequality{\rm )} Let
$\mathcal{T}_h$ be a partition of the domain $\Omega$ into polygons
or polyhedra. Assume that $\mathcal{T}_h$ satisfies all
Assumptions A1-A4 and $p\geq 1$ be any real number. Then, there
exists a constant $C(k)$ such that
\begin{eqnarray}\label{Inverse Inequality00}
\|\nabla\varphi\|_{T,p}\leq
C(k)h^{-1}_T\|\varphi\|_{T,p},\quad\forall T\in\mathcal{T}_h
\end{eqnarray}
for any piecewise polynomial $\varphi$ of degree no more than $k$ on
$\mathcal{T}_h$.
\end{lemma}

\begin{lemma}\label{projection-prop}\rm{(\cite{WY14})}
Let $\mathcal{T}_h$ be a finite element partition of $\Omega$
satisfying the shape regularity assumptions specified in \cite{WY14}
and $w\in H^{k+1}(\Omega)$. Then, for $0 \leq s\leq 1$ we have
\begin{eqnarray}\label{w estimate}
\sumT h_T^{2s} \| w-Q_0 w\|_{T,s}^2 &\leq& C h^{2(k+1)} \|w\|_{k+1}^2,
 \\\label{gradient w estimate}
\sumT h_T^{2s} \|\nabla w-\dQ_h(\nabla w)\|_{T,s}^2 &\leq& C h^{2k} \|w\|_{k+1}^2,
\end{eqnarray}
where $C$ denotes a generic constant independent of mesh size $h$ and the functions in the estimates.
\end{lemma}

Suppose $w\in H^{1+k_1}(\Omega)$ and let $k_0=\min\{k,k_1\}$. With the tools above we can give
the estimates for $\ell(w,v)$ and $s(Q_h w,v)$ as follows.
\begin{lemma}\label{norm-equi1}
For each element $T\in\T_h$, we have
\begin{eqnarray*}
\|\nabla v_0\|_T\le C(\|\nabla_w v\|_T+h_T^{-\frac12}\|Q_b v_0-v_b\|_{\partial T}),\quad
\forall v\in V_h.
\end{eqnarray*}
Furthermore, there is
\begin{eqnarray*}
h^{-\frac12}_T\|v_0-v_b\|_{\partial T}
&\le&C(\|\nabla v_0\|_T+h^{-\frac12}_T
\|Q_b v_0-v_b\|_{\partial T}),
\\
&\le&C(\|\nabla_w v\|_T+h^{-\frac12}_T
\|Q_b v_0-v_b\|_{\partial T}).
\end{eqnarray*}
\end{lemma}
\begin{proof}
From the trace inequality (\ref{Trace inequality00}) and the definition of the
weak gradient operator (\ref{def-wgradient}), we have following inequalities for any $v\in V_h$,
\begin{eqnarray*}
(\nabla v_0,\nabla v_0)_T &=& (\nabla v_0,\nabla_w v)_T+
\langle Q_b v_0-v_b,\nabla v_0\cdot\bn\rangle_{\partial T},
\\
&\le& \|\nabla v_0\|_T\|\nabla_w v\|_T+h_T^{-\frac12}\|Q_b v_0-v_b\|_{\partial T}
h_T^\frac12\|\nabla v_0\|_{\partial T}
\\
&\le& \|\nabla v_0\|_T\|\nabla_w v\|_T+\|\nabla v_0\|_T
h_T^{-\frac12}\|Q_b v_0-v_b\|_{\partial T},
\end{eqnarray*}
which implies that
\begin{eqnarray*}
\|\nabla v_0\|_T\le C(\|\nabla_w v\|_T+h_T^{-\frac12}\|Q_b v_0-v_b\|_{\partial T}).
\end{eqnarray*}
Applying the Poincar\'{e} inequality, we can obtain
\begin{eqnarray*}
h^{-\frac12}_T\|v_0-v_b\|_{\partial T}
&\le& h^{-\frac12}_T\|v_0-Q_b v_0\|_{\partial T}+h^{-\frac12}_T\|Q_b v_0-v_b\|_{\partial T}
\\
&\le& C(\|\nabla v_0\|_T+h^{-\frac12}_T\|Q_b v_0-v_b\|_{\partial T})
\\
&\le& C(\|\nabla_w v\|_T+h^{-\frac12}_T\|Q_b v_0-v_b\|_{\partial T}),
\end{eqnarray*}
which completes the proof.
\end{proof}
\begin{lemma}\label{remainder}
For any $v\in V_h$ and $w\in H^{1+k_1}(\Omega)$,  the following
estimates hold true,
\begin{eqnarray*}
|s(Q_h w,v)|&\le& Ch^{k_0+\frac\eps 2}\|w\|_{k_0+1}\trb v,
\\
|\ell(w,v)|&\le& Ch^{k_0-\frac\eps 2}\|w\|_{k_0+1}\trb v,
\end{eqnarray*}
where $k_0=\min\{k,k_1\}$.
\end{lemma}
\begin{proof}
From the Cauchy-Schwarz inequality and Lemma \ref{norm-equi1}, we can obtain
\begin{eqnarray*}
|s(Q_h w,v)|&=&\left|\sumT h_T^{-1+\eps}\langle Q_bQ_0 w-Q_b w,Q_b v_0-v_b
\rangle_{\partial T}\right|
\\
&\le&C\left(\sumT h_T^{-1+\eps}\|Q_0 w-w\|_{\partial T}^2\right)^\frac12
\left(\sumT h_T^{-1+\eps}\|Q_b v_0-v_b\|_{\partial T}^2\right)^\frac12
\\
&\le& Ch^{k_0+\frac\eps 2}\|w\|_{k_0+1}\trb v.
\end{eqnarray*}
Similarly, for the second term we can derive that
\begin{eqnarray*}
|\ell(w,v)|&=&\left|\sumT\langle(\nabla w-\dQ_h\nabla w)\cdot\bn,
 v_0-v_b\rangle_{\partial T}\right|\\
 &\le& C\left(\sumT h_T^{1-\eps}\|\nabla w-\dQ_h\nabla w\|_{\partial T}^2\right)^\frac12
\left(\sumT h_T^{-1+\eps}\|v_0-v_b\|_{\partial T}^2\right)^\frac12\\
&\le& Ch^{k_0-\frac\eps 2}\|w\|_{k_0+1}\trb v,
\end{eqnarray*}
which completes the proof.
\end{proof}

\subsection{Error estimates}
With the error equation (\ref{error-eqn}) and the estimates derived in
Lemma \ref{remainder}, we can get the following error estimate for the
weak Galerkin method.
\begin{theorem}\label{err-est}
Assume the exact solution of (\ref{Poisson-eq1}),
$u\in H^{1+k_1}(\Omega)$, and $u_h$ is the numerical solution of the weak Galerkin
scheme (\ref{WG-scheme1}). Denote $k_0=\min\{k,k_1\}$, then the following
estimates hold true,
\begin{eqnarray}\label{err-est1}
\trb{Q_h u-u_h}\le Ch^{k_0-\frac\eps 2}\|u\|_{k_0+1},\\\label{err-est2}
\trb{Q_h u-u_h}_1\le Ch^{k_0-\eps}\|u\|_{k_0+1}.
\end{eqnarray}
\end{theorem}
\begin{proof}
Taking $v=e_h$ in (\ref{error-eqn}) and it follows that
\begin{eqnarray*}
\trb{e_h}^2&=&\ell(u,e_h)+s(Q_h u,e_h)\\
&\le& Ch^{k_0-\frac\eps 2}\|w\|_{k_0+1}\trb {e_h}+Ch^{k_0-\frac\eps 2}
\|w\|_{k_0+1}\trb {e_h}\\
&\le& Ch^{k_0-\frac\eps 2}\|w\|_{k_0+1}\trb {e_h}.
\end{eqnarray*}
From the definition of $\trb\cdot_1$, we can easily get that
when $h$ is small,
\begin{eqnarray*}
\trb{Q_h u-u_h}_1&\le& Ch^{-\frac\eps 2}\trb{Q_h u-u_h}
\le Ch^{k_0-\eps}\|u\|_{k_0+1},
\end{eqnarray*}
which completes the proof.
\end{proof}

Using a standard dual argument, which is similar to the technique applied
in \cite{WY2}, and then we can obtain the following $L^2(\Omega)$ error
estimate.
\begin{theorem}\label{err-estL2}
Assume the exact solution of (\ref{Poisson-eq1}),
$u\in H^{1+k_1}(\Omega)$, and $u_h$ is the numerical solution of the weak Galerkin
scheme (\ref{WG-scheme1}). In addition, assume the dual problem has
$H^2(\Omega)$-regularity. Denote $k_0=\min\{k,k_1\}$, then the following
estimate holds true
\begin{eqnarray}\label{L2_Estimate_Solution}
\|Q_0 u-u_0\|\le Ch^{k_0+1-\eps}\|u\|_{k_0+1}.
\end{eqnarray}
\end{theorem}
\section{Error estimates for the eigenvalue problem}
In this section, we turn back to the approximation of the eigenvalue
problem (\ref{problem-eq}). Denote $V_0=H_0^1(\Omega)$,
and define the sum space $V=V_0+V_h$. Now we introduce the following
semi-norm on $V$ that
\begin{eqnarray*}
\|w\|_V^2= \sumT \Big(\|\nabla w_0\|_T^2+ h_T^{-1}\|Q_b w_0-w_b\|^2_{\partial
T}\Big).
\end{eqnarray*}
We claim that $\|\cdot\|_V$ indeed defines a norm on $V$. For any
$w\in V_0$, if $w_b$ is defined in the sense of trace, we shall show that
$\|w\|_V$ is equivalent to $|w|_1$, which defines a norm on $V_0$.
\begin{lemma}\label{norm-equi3}
$\|\cdot\|_V$ is equivalent to $|\cdot|_1$ on $V_0$.
\end{lemma}
\begin{proof}
It is obvious that for any $w\in V_0$,
\begin{eqnarray*}
|w|_1 \le \|w\|_V.
\end{eqnarray*}
Then we only need to show that
\begin{eqnarray*}
\sumT h_T^{-1}\|Q_b w-w\|_{\partial T}^2 \le C\|\nabla w\|^2.
\end{eqnarray*}
To this end, denote by $Q_c$ the $L^2$ projection onto $P_0(T)$,
and it follows the trace inequality (\ref{Trace inequality00}) and
the Poincare's inequality that
\begin{eqnarray*}
\sumT h_T^{-1}\|Q_b w-w\|_{\partial T}^2
&\le& \sumT h_T^{-1}\|w-Q_c w\|_{\partial T}^2
\\
&\le& C\sumT \left(h_T^{-2}\|w-Q_c w\|_{T}^2 +
\|\nabla(w-Q_c w)\|_{T}^2\right)
\\
&\le& C\|\nabla w\|^2,
\end{eqnarray*}
which completes the proof.
\end{proof}

 As to the space $V_h$, we have the following equivalence
lemma.
\begin{lemma}\label{norm-equi}\rm{(\cite{WY2})}
There exists two constants $C_1$ and $C_2$ such that for any $w\in V_h$,
we have
\begin{eqnarray}\label{norm-equi2}
C_1\|w\|_V\le \trb w_1 \le C_2\|w\|_V,
\end{eqnarray}
i.e. $\|\cdot\|_V$ and $\trb\cdot_1$ are equivalent on $V_h$.
\end{lemma}
\begin{proof}
In order to prove the equivalence, we just need to verify the following
inequalities that for any $w\in V_h$,
\begin{eqnarray}
&&\|\nabla w_0\|_T
\le C(\|\nabla_w w\|_T+h_T^{-\frac12}\|Q_b w_0-w_b\|_{\partial T}),
\label{norm-neq1}\\
&&\|\nabla_w w\|_T
\le C(\|\nabla w_0\|_T+h_T^{-\frac12}\|Q_b w_0-w_b\|_{\partial T}).\label{norm-neq2}
\end{eqnarray}
The inequality (\ref{norm-neq1}) has been proved in Lemma \ref{norm-equi1}.
For handling the inequality (\ref{norm-neq2}), we use the definition
of the weak gradient to get that
\begin{eqnarray*}
(\nabla_w w,\nabla_w w)_T&=&(\nabla_w w,\nabla w_0)-\langle
w_0-w_b,\nabla_w w\cdot\bn\rangle_{\partial T}
\\
&\le&C(\|\nabla_w w\|_T\|\nabla w_0\|_T-\|\nabla_w w\|_Th^{-\frac12}_T
\|w_0-w_b\|_{\partial T}).
\end{eqnarray*}
Then we can derive from Lemma \ref{norm-equi1} that
\begin{eqnarray*}
\|\nabla_w w\|_T&\le& C(\|\nabla w_0\|_T+h_T^{-\frac12}\|w_0-w_b\|_{\partial T})
\\
&\le&C(\|\nabla w_0\|_T+h_T^{-\frac12}\|Q_b w_0-w_b\|_{\partial T}),
\end{eqnarray*}
which completes the proof.
\end{proof}

Now we define two operators $K$ and $K_h$ as follows
\begin{eqnarray*}
&&K:\ L^2(\Omega)\rightarrow V_0 \text{ satisfying } a(K f,v)=(f,v),\quad\forall v\in V_0,
\\
&&K_h:L^2(\Omega)\rightarrow V_h \text{ satisfying } a_w(K_h f,v)=(f,v_0),\quad\forall v\in V_h.
\end{eqnarray*}

The following lemmas show that the finite element
space $V_h$ and the discrete solution operator $K_h$
are approximations of $V$ and $K$.
\begin{lemma}\label{interpolation}
Suppose $w\in V_0\cap H^{1+k_1}(\Omega)$, then we have
\begin{eqnarray*}
\|w-Q_h w\|_V &\le& C h^{k_1}\|w\|_{k_1+1}.
\end{eqnarray*}
\end{lemma}
\begin{proof}
From the trace
inequality (\ref{Trace inequality00}) and Lemma
\ref{projection-prop} we have
\begin{eqnarray*}
&& \|Q_h w-w\|_V
\\
&\le& C\left(\sumT\Big(\|\nabla(Q_0 w-w)\|_T^2+ h^{-1}_T\|(Q_0 w-w)-
Q_b (Q_0w -w)\|
_{\partial T}^2\Big)\right)^{\frac12}
\\
&\le& C\left(\sumT\Big(\|\nabla(Q_0 w-w)\|_T^2+ h^{-1}_T\|Q_0 w- w\|
_{\partial T}^2\Big)\right)^{\frac12}
\\
&\le& C\left(\sumT\Big(\|\nabla(Q_0 w-w)\|_T^2+h^{-1}
\|Q_0 w- w\|_{\partial T}^2\Big)\right)^{\frac12}
\\
&\le& C h^{k_1}\|w\|_{k_1+1},
\end{eqnarray*}
which completes the proof.
\end{proof}

As we know, we can extend the operators $K$ and $K_h$ to the operators from
 $L^2(\Omega)$ to $V$
which will not change the non-zero spectrums of the operators $K$ and $K_h$.
\begin{lemma}\label{operator-approx}
The operators $K$ and $K_h$ have the following estimate
\begin{eqnarray*}
\lim_{h\rightarrow 0}\|K_h-K\|_V=0,
\end{eqnarray*}
where $\|\cdot\|_V$ denote the operator norm from $V$ to $V$.
\end{lemma}
\begin{proof}
Since $V$ is a Hilbert space, it is equivalent to verify that
\begin{eqnarray*}
\lim_{h\rightarrow 0}\sup_{\|f\|_V=1}\|K f-K_h f\|_V=0.
\end{eqnarray*}

For any $f\in V$ with $\|f\|_V=1$, suppose $u=Kf$ and $u_h=K_hf$.
From the error estimate (\ref{err-est2}) and the regularity of the
Poisson's equation, we have
\begin{eqnarray*}
\trb{Q_h u-u_h}_1\le Ch^{1-\eps}\|u\|_2\le
Ch^{1-\eps}\|f\|\le Ch^{1-\eps}\|f\|_V.
\end{eqnarray*}
Then the equivalence Lemma \ref{norm-equi} implies that
\begin{eqnarray*}
\|Q_h u-u_h\|_V\le Ch^{1-\eps}\|f\|_V.
\end{eqnarray*}
Moreover, by letting $k_1=1$ in Lemma \ref{interpolation}
we can obtain that
\begin{eqnarray*}
\|u-Q_h u\|_V \le Ch\|u\|_2 \le Ch\|f\|_V.
\end{eqnarray*}
It follows the triangle inequality that
\begin{eqnarray*}
\|K f - K_h f\|_V=\|u-u_h\|_V\le Ch^{1-\eps}\|f\|_V.
\end{eqnarray*}
Notice that $0\le\eps<1$, which completes the proof.
\end{proof}

\begin{lemma}
The operator $K_h: V\mapsto V$ is compact.
\end{lemma}
\begin{proof}
Denote $\tilde K_h$ the restriction of $K_h$ on $V_h$. Since $V_h$
is finite dimensional, $\tilde K_h$ is compact. Notice that $(Q_0 f,
v_0)=(f,v_0)$, so $K_h=\tilde K_h Q_h$. In order to prove that $K_h$
is compact, we just need to verify that $Q_h$ is bounded.

For any $w\in V_h$, $Q_h w=w$. For $w\in V_0$, we can conclude from
Lemma \ref{interpolation} that
\begin{eqnarray*}
\|Q_h w\|_V \le \|Q_h w-w\|_V+\|w\|_V \le C\|w\|_1+\|w\|_V
\le C\|w\|_V,
\end{eqnarray*}
which completes the proof.
\end{proof}

Now we review some notations in the spectral approximation theory.
We denote by $\sigma(K)$ the spectrum of $K$, and by $\rho(K)$ the
resolvent set. $R_z(K)=(z-K)^{-1}$ represents the resolvent operator.
Let $\mu$ be a nonzero eigenvalue of $T$ with algebraic multiplicities
$m$. Let $\Gamma$ be a circle in the complex plane centered at $\mu$
which lies in $\rho(K)$ and encloses no other points of
$\sigma(K)$. The corresponding spectral projection is
\begin{eqnarray*}
E=E(\mu)=\frac{1}{2\pi {\rm i}}\int_\Gamma R_z(K) dz.
\end{eqnarray*}
$R(E)$ represents the range of $E$, which is the space of generalized
eigenvectors.

For a Banach space $X$ and its closed subspaces $M$ and $N$, define
the distances as follows that
\begin{eqnarray*}
&&{\rm dist}(x,N)=\inf_{y\in N}\|x-y\|,~
\delta(M,N)=\sup_{\substack{x\in M,\\\|x\|=1}} {\rm dist}(x,N),
\\
&&\hat\delta(M,N)=\max\{\delta(M,N),\delta(N,M)\}.
\end{eqnarray*}
\begin{lemma}
Suppose the eigenvectors of \ $K$ have $k_1$-regularity, i.e. $w\in
H^{1+k_1}(\Omega)$ for any $w\in R(E)$. Denote $k_0=\min\{k_1,k\}$,
then the following estimate holds true,
\begin{eqnarray*}
\|K - K_h|_{R(E)}\|_V\le Ch^{k_0-\eps}.
\end{eqnarray*}
\end{lemma}
\begin{proof}
Suppose $w\in R(E)$ with $\|w\|_V=1$. Similar to the proof of Lemma
\ref{operator-approx}, from Theorem \ref{err-est} we have
\begin{eqnarray*}
\|Q_h K w-K_h w\|_V\le Ch^{k_0-\eps}\|T w\|_{k_0+1}.
\end{eqnarray*}
From Lemma \ref{interpolation} we can obtain
\begin{eqnarray*}
\|K w-Q_h K w\|_V\le Ch^{k_0}\|K w\|_{k_0+1},
\end{eqnarray*}
which implies
\begin{eqnarray*}
\|K w - K_h w\|_V\le Ch^{k_0-\eps}\|K w\|_{k_0+1}\le Ch^{k_0-\eps}\|w\|_{k_0+1}.
\end{eqnarray*}
Since $R(E)$ is finite dimensional, there is a uniform upper bound
for $\|w\|_{k_0+1}$, where $w\in R(E)$ with $\|w\|_V=1$,
which completes the proof.
\end{proof}

For the symmetry of $a(\cdot,\cdot)$ and $a_w(\cdot,\cdot)$,
$K$ and $K_h$ are self-adjoint. In addition, if we change the $\|\cdot\|_V$
norm to $L^2(\Omega)$ norm, all the conclusions in this section can be
interpreted trivially. Then from the theory in \cite{IBEig,Chatelin}, we can derive the following
estimates.
\begin{theorem}\label{Error_Theorem_Eigenpair}
Suppose $\lambda_{j,h}$ is the $j$-th eigenvalue of (\ref{WG-scheme}) and $u_{j,h}$ is the corresponding eigenvector.
There exist an exact eigenvalue $\lambda_j$ and the corresponding exact eigenfunction $u_j$ such that
the following error estimates hold
\begin{eqnarray}\label{eig-est1}
&&|\lambda_j-\lambda_{j,h}|\le Ch^{2k_0-2\eps}\|u_j\|_{k_0+1},
\\\label{eig-est2}
&&\|u_j-u_{j,h}\|_V \le Ch^{k_0-\eps}\|u_j\|_{k_0+1},
\\\label{eig-est3}
&&\|u_j-u_{j,h}\|\le Ch^{k_0+1-\eps}\|u_j\|_{k_0+1},
\end{eqnarray}
when $u_j\in H^{1+k_1}(\Omega)$ and $k_0=\min\{k_1,k\}$.
\end{theorem}

\section{Lower bounds}
In this section, we shall demonstrate that the approximate eigenvalue
$\lambda_h$ generated by (\ref{WG-scheme}) is an asymptotic lower bound of
$\lambda$. About the topic of lower bound of the eigenvalues, please refer to \cite{HuHuangLin,Liu,LuoLinXie,ZhangYangChen}
and the references cited therein.
 In this section, the parameter $\eps$ is required
to be positive, i.e. $0<\eps<1$.
\begin{lemma}\label{expansion}
Suppose $(\lambda,u)$ is the solution of (\ref{vartion-form}), and $(\lambda_h,u_h)$ is the solution of
(\ref{WG-scheme}). Then we have the following expansion that
for any $v\in V_h$,
\begin{eqnarray*}
\lambda-\lambda_h&=&\|\nabla u-\nabla_w u\|^2+s(u_h-v,u_h-v)-\lambda_h
\|u_0-v_0\|^2-\lambda_h(\|u_0\|^2-\|v_0\|^2)
\\&&+2(\nabla u-
\nabla_w v,\nabla_w u_h)-s(v,v).
\end{eqnarray*}
\end{lemma}
\begin{proof}
First, we have the following formulas
\begin{eqnarray*}
&&a(u,u)=\|\nabla u\|^2=\lambda\|u\|^2,
\\
&&a_w(u_h,u_h)=\|\nabla_w u_h\|^2+s(u_h,u_h)=\lambda_h\|u_0\|^2,
\\
&&\|u\|=\|u_0\|=1.
\end{eqnarray*}
Mimicking the expansion in \cite{AD} and we have the following
expansion
\begin{eqnarray*}
&&(\nabla u-\nabla_w u_h,\nabla u-\nabla_w u_h)
\\
&=&(\nabla u,\nabla u)+(\nabla_w u_h,\nabla_w u_h)-2(\nabla u,
\nabla_w u_h)
\\
&=&\lambda+\lambda_h-2(\nabla u,\nabla_w u_h)-s(u_h,u_h)
\\
&=&\lambda+\lambda_h-2(\nabla u-\nabla_w v,\nabla_w u_h)
-2(\nabla_w v,\nabla_w u_h)-s(u_h,u_h)
\\
&=&\lambda+\lambda_h-2(\nabla u-\nabla_w v,\nabla_w u_h)
-2\lambda_h(u_0,v_0)+2s(u_h,v)-s(u_h,u_h)
\\
&=&\lambda+\lambda_h-2(\nabla u-\nabla_w v,\nabla_w u_h)
+\lambda_h(u_0-v_0,u_0-v_0)-\lambda_h(u_0,u_0)-\lambda_h(v_0,v_0)
\\
&&+2s(u_h,v)-s(u_h,u_h)
\\
&=&\lambda-\lambda_h-2(\nabla u-\nabla_w v,\nabla_w u_h)
+\lambda_h(u_0-v_0,u_0-v_0)+\lambda_h((u_0,u_0)-(v_0,v_0))
\\
&&+2s(u_h,v)-s(u_h,u_h).
\end{eqnarray*}
Rearranging the above formula and it follows that
\begin{eqnarray*}
\lambda-\lambda_h&=&\|\nabla u-\nabla_w u\|^2+s(u_h-v,u_h-v)-\lambda_h
\|u_0-v_0\|^2-\lambda_h(\|u_0\|^2-\|v_0\|^2)
\\&&+2(\nabla u-
\nabla_w v,\nabla_w u_h)-s(v,v),
\end{eqnarray*}
which completes the proof.
\end{proof}
\begin{lemma}\label{convergence-rate}\rm{(\cite{LX1})}
The following lower bound for the convergence rate holds for the exact eigenfunction $u$ of
the eigenvalue problem (\ref{vartion-form})
\begin{eqnarray*}
\|\nabla u-\dQ_h \nabla u\|\ge Ch^{2k}.
\end{eqnarray*}
\end{lemma}
\begin{theorem}
Let $\lambda_j$ and $\lambda_{j,h}$ be the $j$-th exact eigenvalue
and its corresponding weak Galerkin numerical approximation. Assume
the corresponding eigenvector $u\in H^{1+k_1}(\Omega)$. Denote
$k_0=\min\{k,k_1\}$. Then if the mesh size $h$ is small enough,
there exists a constant C such that
\begin{eqnarray*}
0\le \lambda_j-\lambda_{j,h}\le Ch^{2k_0-2\eps}\|u\|_{k_0+1}.
\end{eqnarray*}
\end{theorem}
\begin{proof}
Take $v=Q_h u$ in Lemma \ref{expansion}. From the commutative
property in Lemma \ref{commu-prop}, there is
\begin{eqnarray*}
\nabla_w v=\dQ_h \nabla u,
\end{eqnarray*}
and it follows that
\begin{eqnarray*}
\lambda-\lambda_h&=&\|\nabla u-\nabla_w u\|^2+s(u_h-v,u_h-v)-\lambda_h
\|u_0-v_0\|^2-\lambda_h(\|u_0\|^2-\|v_0\|^2)
\\&&+2(\nabla u-
\nabla_w v,\nabla_w u_h)-s(v,v)
\\
&=&\|\nabla u-\dQ_h \nabla u\|^2+\trb{Q_hu-u_h}^2
-\lambda_h\|Q_0 u-u_0\|^2-\lambda_h(\|u_0\|^2-\|Q_0 u\|^2)
\\&&+2(\nabla u-
\dQ_h\nabla u,\nabla_w u_h)-s(Q_h u,Q_h u).
\end{eqnarray*}

Since $\nabla_w u_h\in [P_{k-1}(T)]^2$, we can obtain
\begin{eqnarray*}
(\nabla u-\dQ_h\nabla u,\nabla_w u_h)=0.
\end{eqnarray*}
From the error estimate (\ref{eig-est2}) and (\ref{eig-est3}),
we have
\begin{eqnarray*}
\trb{Q_h u-u_h}^2\le\|Q_h u-u_h\|_V^2\le Ch^{2k_0-2\eps}\|u\|_{k_0+1}
\end{eqnarray*}
and
\begin{eqnarray*}
\|Q_0 u-u_0\|^2\le Ch^{2k_0+2-2\eps}\|u\|_{k_0+1}.
\end{eqnarray*}
Also, it follows the property of projection in Lemma \ref{projection-prop}
that
\begin{eqnarray*}
\|Q_0 u-u_0\|^2&=&(u_0+Q_0 u,u_0-Q_0 u)
\\
&=&((u-u_0)+(u-Q_0 u),(u-u_0)-(u-Q_0 u))
\\
&=&\|u-u_0\|^2-\|u-Q_0 u\|^2
\\
&\le& Ch^{2k_0+2-2\eps}\|u\|_{k_0+1},
\\
s(Q_h u,Q_h u) &=& \sumT h_T^{1+\eps}\|Q_b Q_0 u -Q_b u\|_{\partial T}^2
\\
&\le& \sumT h_T^{1+\eps}\|Q_0 u -u\|_{\partial T}^2
\\
&\le& Ch^{2k_0+\eps}\|u\|_{k_0+1}.
\end{eqnarray*}
From Lemma \ref{convergence-rate}, we know the terms $\lambda_h\|u_0-Q_0 u\|^2$,
$\lambda_h(\|u_0\|^2-\|Q_0 u\|^2))$, $(\nabla u-\dQ_h\nabla u,\nabla_w u_h)$,
and $s(Q_hu, Q_hu)$ are of higher order than $\|\nabla u-\dQ_h\nabla u\|^2$,
so that
\begin{eqnarray*}
Ch^{2k_0}\|u\|_{k_0+1}\le \trb{Q_h u-u_h}^2+\|\nabla u-\dQ_h\nabla u\|^2\le Ch^{2k_0-2\eps}\|u\|_{k_0+1}
\end{eqnarray*}
is the dominant term, which completes the proof.
\end{proof}

\begin{remark}
 In the next section, the numerical results show that the convergence
rates in fact tend to $2k_0-\eps$. On the other hand, the numerical eigenvalue $\lambda_h$
is still a lower bound even if $\eps=0$.
\end{remark}

\section{Numerical Experiments}
In this section, we shall present some numerical results for the
weak Galerkin method analyzed in the previous sections.

\subsection{Unit square domain} In the first example, we consider the problem
(\ref{problem-eq}) on the square domain
$\Omega=(0,1)^2$. It has the analytic solution
\begin{eqnarray*}
\lambda= (m^2+n^2)\pi^2,\ \ \ \ \
&&u=\sin(m\pi x)\sin(n\pi y),
\end{eqnarray*}
where $m$, $n$ are arbitrary integers. The first four eigenvalues
are $\lambda_1=2\pi^2$, $\lambda_2=5\pi^2$, $\lambda_3=8\pi^2$ and
$\lambda_4=10\pi^2$, where $\lambda_2$ and $\lambda_4$ have $2$
algebraic and geometric multiplicities.

The uniform mesh is applied in the following examples, and $h$
denote the mesh size. Different choices of the parameter $\eps$ and
the degree of polynomial $k$ are presented.
The corresponding numerical results for the first four eigenvalues are showed in
Tables \ref{Tables_Exam1_1}-\ref{Tables_Exam1_6}. From these tables, we can find the
weak Galerkin method can also give the reasonable numerical approximations.
Furthermore, the choice of $\eps$ can really affect the convergence order which means the convergence results in
 Theorem \ref{Error_Theorem_Eigenpair} are also reasonable. The numerical results included in
  Tables \ref{Tables_Exam1_1}-\ref{Tables_Exam1_6} shows the eigenvalue approximation $\lambda_{j,h}$
 is a lower bound of the exact eigenvalue $\lambda_j$.
 Tables \ref{Tables_Exam1_7}-\ref{Tables_Exam1_12} shows the convergence
 behavior of the eigenfunction approximations
 which reveal the convergence results in Theorem \ref{Error_Theorem_Eigenpair}.

{\footnotesize
\begin{table}[]
\centering \caption{Convergence rates for $\eps=0.1$ and $k=1$.}\label{Tables_Exam1_1}
\begin{tabular}{|c|c|c|c|c|c|c|}
\hline
$h$ &  1/4 &  1/8     & 1/16  &  1/32   & 1/64 & 1/128 \\
\hline \hline
$\lambda_1-\lambda_{1,h}$   &  4.6914e+0 & 1.5050e+0 & 4.2473e-1 & 1.1518e-1 & 3.0896e-2 & 8.2640e-3 \\
\hline
order                       &             & 1.6403     & 1.8251     & 1.8826     &  1.8984    & 1.9025     \\
\hline
$\lambda_2-\lambda_{2,h}$   & 2.2610e+1  & 8.8734e+0 & 2.7033e+0 & 7.4857e-1 & 2.0152e-1 & 5.3843e-2   \\
\hline
order    &     & 1.3494 & 1.7147 & 1.8525 & 1.8932 & 1.9041   \\
\hline
$\lambda_3-\lambda_{3,h}$ & 4.4453e+1 & 1.9725e+1 & 6.3969e+0 & 1.8123e+0 & 4.9211e-1 & 1.3206e-1   \\
\hline
order  &   & 1.1722 & 1.6246 & 1.8196 & 1.8808 & 1.8977   \\
\hline
$\lambda_4-\lambda_{4,h}$ & 6.4193e+1 & 3.1210e+1 & 1.0638e+1 & 3.0563e+0 & 8.3024e-1 & 2.2206e-1 \\
\hline
order   &   & 1.0404 & 1.5528 & 1.7993 & 1.8802 & 1.9026  \\
\hline
\end{tabular}
\end{table}
}

{\footnotesize
\begin{table}[]
\centering \caption{Convergence rates for $\eps=0.05$ and $k=1$.}\label{Tables_Exam1_2}
 \begin{tabular}{|c|c|c|c|c|c|c|}
\hline
$h$                         &  1/4        &  1/8        & 1/16      &  1/32       & 1/64       & 1/128     \\
\hline \hline
$\lambda_1-\lambda_{1,h}$   & 4.5174e+0   & 1.3948e+0   & 3.7944e-1 & 9.9389e-2   & 2.5770e-2  & 6.6642e-3 \\
\hline
order                       &             &1.6955       & 1.8781    & 1.9327      & 1.9474     & 1.9512     \\
\hline
$\lambda_2-\lambda_{2,h}$   & 2.2044e+1  & 8.3229e+0    & 2.4379e+0 & 6.5153e-1   & 1.6962e-1  & 4.3857e-2  \\
\hline
order                       &            & 1.4052       & 1.7714    & 1.9037      & 1.9415     & 1.9514     \\
\hline
$\lambda_3-\lambda_{3,h}$   & 4.3486e+1  & 1.8531e+1    & 5.7514e+0 & 1.5674e+0   & 4.1076e-1  & 1.0652e-1  \\
\hline
order                       &            & 1.2306       & 1.6880    & 1.8756      & 1.9320     & 1.9471    \\
\hline
$\lambda_4-\lambda_{4,h}$   & 6.3226e+1  & 2.9650e+1    & 9.6842e+0 & 2.6795e+0   & 7.0366e-1  & 1.8219e-1  \\
\hline
order                       &            & 1.0925       &  1.6143   & 1.8537      & 1.9290     & 1.9494   \\
\hline
\end{tabular}
\end{table}
}

{\footnotesize
\begin{table}[]
\centering \caption{Convergence rates for $\eps=0$ and $k=1$.}\label{Tables_Exam1_3}
 \begin{tabular}{|c|c|c|c|c|c|c|}
\hline
$h$                         &  1/4        &  1/8        & 1/16      &  1/32       & 1/64       & 1/128     \\
\hline \hline
$\lambda_1-\lambda_{1,h}$ & 4.3486e+0 & 1.2926e+0 & 3.3916e-1 & 8.5854e-2 & 2.1531e-2 & 5.3870e-3 \\
\hline
order &       & 1.7503  & 1.9302  & 1.9820  & 1.9955  & 1.9989  \\
\hline
$\lambda_2-\lambda_{2,h}$ & 2.1485e+1 & 7.8051e+0 & 2.2003e+0 & 5.6820e-1 & 1.4323e-1 & 3.5882e-2 \\
\hline
order &       & 1.4608  & 1.8267  & 1.9532  & 1.9881  & 1.9970  \\
\hline
$\lambda_3-\lambda_{3,h}$ & 4.2518e+1 & 1.7394e+1 & 5.1703e+0 & 1.3566e+0 & 3.4342e-1 & 8.6124e-2 \\
\hline
order &       & 1.2895  & 1.7503  & 1.9302  & 1.9820  & 1.9955  \\
\hline
$\lambda_4-\lambda_{4,h}$ & 6.2257e+1 & 2.8157e+1 & 8.8226e+0 & 2.3548e+0 & 5.9881e-1 & 1.5035e-1 \\
\hline
order &       & 1.1447  & 1.6742  & 1.9056  & 1.9754  & 1.9938  \\
\hline
\end{tabular}
\end{table}
}

{\footnotesize
\begin{table}[]
\centering \caption{Convergence rates for $\eps=0.1$ and $k=2$.}\label{Tables_Exam1_4}
 \begin{tabular}{|c|c|c|c|c|c|c|}
\hline
$h$                         &  1/4        &  1/8        & 1/16      &  1/32       & 1/64       & 1/128     \\
\hline \hline
$\lambda_1-\lambda_{1,h}$ & 2.2414e-1 & 1.4066e-2 & 9.2552e-4 & 6.1762e-5 & 4.1371e-6 & 2.7718e-7 \\ \hline
order &       & 3.9941  & 3.9259  & 3.9055  & 3.9000  & 3.8997  \\ \hline
$\lambda_2-\lambda_{2,h}$ & 3.4702e+0 & 1.9914e-1 & 1.2641e-2 & 8.3435e-4 & 5.5667e-5 & 3.7246e-6 \\ \hline
order &       & 4.1232  & 3.9776  & 3.9213  & 3.9058  & 3.9016  \\ \hline
$\lambda_3-\lambda_{3,h}$ & 1.5809e+1 & 9.7847e-1 & 6.0649e-2 & 3.9786e-3 & 2.6525e-4 & 1.7761e-5 \\ \hline
order &       & 4.0141  & 4.0120  & 3.9302  & 3.9068  & 3.9006  \\ \hline
$\lambda_4-\lambda_{4,h}$ & 2.7771e+1 & 2.0077e+0 & 1.2222e-1 & 7.9478e-3 & 5.2797e-4 & 3.5281e-5 \\ \hline
order &       & 3.7900  & 4.0380  & 3.9428  & 3.9120  & 3.9035  \\ \hline
\end{tabular}
\end{table}
}

\begin{table}[]
\centering \caption{Convergence rates for $\eps=0.05$ and $k=2$.}\label{Tables_Exam1_5}
\begin{tabular}{|c|c|c|c|c|c|c|c|c|}
\hline
$h$      &  1/4        &  1/8        & 1/16      &  1/32       & 1/64       & 1/128     \\
\hline \hline
$\lambda_1-\lambda_{1,h}$ & 2.1062e-1 & 1.2818e-2 & 8.1597e-4 & 5.2613e-5 & 3.4037e-6 & 2.2032e-7 \\ \hline
order &       & 4.0384  & 3.9735  & 3.9550  & 3.9503  & 3.9494  \\ \hline
$\lambda_2-\lambda_{2,h}$ & 3.2288e+0 & 1.8086e-1 & 1.1158e-2 & 7.1301e-4 & 4.5989e-5 & 2.9733e-6 \\ \hline
order &       & 4.1580  & 4.0188  & 3.9680  & 3.9546  & 3.9511  \\ \hline
$\lambda_3-\lambda_{3,h}$ & 1.4735e+1 & 8.8125e-1 & 5.3250e-2 & 3.3846e-3 & 2.1813e-4 & 1.4109e-5 \\ \hline
order &       & 4.0635  & 4.0487  & 3.9757  & 3.9557  & 3.9505  \\ \hline
$\lambda_4-\lambda_{4,h}$ & 2.6186e+1 & 1.8074e+0 & 1.0753e-1 & 6.7866e-3 & 4.3622e-4 & 2.8175e-5 \\ \hline
order &       & 3.8568  & 4.0711  & 3.9859  & 3.9596  & 3.9525  \\ \hline
\end{tabular}
\end{table}

\begin{table}[]
\centering \caption{Convergence rates for $\eps=0$ and $k=2$.}\label{Tables_Exam1_6}
\begin{tabular}{|c|c|c|c|c|c|c|c|c|}
\hline
$h$ &    1/4   &    1/8   &    1/16  &    1/32  &    1/64  &    1/128 \\
\hline \hline
$\lambda_1-\lambda_{1,h}$ & 1.9798e-1 & 1.1680e-2 & 7.1897e-4 & 4.4770e-5 & 2.7957e-6 & 1.7471e-7 \\ \hline
order &       & 4.0833  & 4.0219  & 4.0053  & 4.0012  & 4.0002  \\ \hline
$\lambda_2-\lambda_{2,h}$ & 3.0076e+0 & 1.6440e-1 & 9.8518e-3 & 6.0934e-4 & 3.7989e-5 & 2.3730e-6 \\ \hline
order &       & 4.1933  & 4.0606  & 4.0151  & 4.0036  & 4.0008  \\ \hline
$\lambda_3-\lambda_{3,h}$ & 1.3731e+1 & 7.9457e-1 & 4.6751e-2 & 2.8767e-3 & 1.7911e-4 & 1.1184e-5 \\ \hline
order &       & 4.1111  & 4.0871  & 4.0225  & 4.0055  & 4.0013  \\ \hline
$\lambda_4-\lambda_{4,h}$ & 2.4673e+1 & 1.6299e+0 & 9.4683e-2 & 5.7963e-3 & 3.6040e-4 & 2.2497e-5 \\ \hline
order &       & 3.9201  & 4.1056  & 4.0299  & 4.0075  & 4.0018  \\ \hline
\end{tabular}
\end{table}

{\footnotesize
\begin{table}[]
\centering \caption{Convergence rates for $\eps=0.1$ and $k=1$.}\label{Tables_Exam1_7}
\begin{tabular}{|c|c|c|c|c|c|c|c|c|c|c|c|c|}
\hline
$h$     &    1/4   &    1/8   &    1/16  &    1/32  &    1/64  &    1/128 \\ \hline \hline
$\trb{Q_h u_1-u_{1,h}}$ & 2.1026e+0 & 1.1328e+0 & 5.9471e-1 & 3.1002e-1 & 1.6128e-1 & 8.3847e-2 \\ \hline
order & 0.0000  & 0.8923  & 0.9296  & 0.9398  & 0.9427  & 0.9438  \\ \hline
$\|Q_h u_1-u_{1,h}\|$ & 2.4300e-1 & 6.7500e-2 & 1.8371e-2 & 4.9711e-3 & 1.3428e-3 & 3.6248e-4 \\ \hline
order & 0.0000  & 1.8480  & 1.8775  & 1.8858  & 1.8883  & 1.8893  \\ \hline
$\trb{Q_h u_2-u_{2,h}}$ & 4.1174e+0 & 2.6638e+0 & 1.4918e+0 & 7.9139e-1 & 4.1310e-1 & 2.1467e-1 \\ \hline
order & 0.0000  & 0.6283  & 0.8364  & 0.9146  & 0.9379  & 0.9444  \\ \hline
$\|Q_h u_2-u_{2,h}\|$ & 2.0401e-1 & 5.0309e-2 & 1.2723e-2 & 3.2491e-3 & 8.3336e-4 & 2.1441e-4 \\ \hline
order & 0.0000  & 2.0198  & 1.9833  & 1.9694  & 1.9630  & 1.9586  \\ \hline
$\trb{Q_h u_3-u_{3,h}}$ & 4.1174e+0 & 2.6638e+0 & 1.4918e+0 & 7.9139e-1 & 4.1310e-1 & 2.1467e-1 \\ \hline
order & 0.0000  & 0.6283  & 0.8364  & 0.9146  & 0.9379  & 0.9444  \\ \hline
$\|Q_h u_3-u_{3,h}\|$ & 2.0401e-1 & 5.0309e-2 & 1.2723e-2 & 3.2491e-3 & 8.3336e-4 & 2.1441e-4 \\ \hline
order & 0.0000  & 2.0198  & 1.9833  & 1.9694  & 1.9630  & 1.9586  \\ \hline
$\trb{Q_h u_4-u_{4,h}}$ & 8.7353e+0 & 4.3751e+0 & 2.3571e+0 & 1.2372e+0 & 6.4476e-1 & 3.3534e-1 \\ \hline
order & 0.0000  & 0.9975  & 0.8923  & 0.9300  & 0.9402  & 0.9431  \\ \hline
$\|Q_h u_4-u_{4,h}\|$ & 1.1748e+0 & 2.6252e-1 & 7.2940e-2 & 1.9847e-2 & 5.3688e-3 & 1.4497e-3 \\ \hline
order & 0.0000  & 2.1620  & 1.8477  & 1.8778  & 1.8863  & 1.8888  \\ \hline
$\trb{Q_h u_5-u_{5,h}}$ & 7.6747e+0 & 4.9834e+0 & 2.9631e+0 & 1.6029e+0 & 8.4057e-1 & 4.3693e-1 \\ \hline
order & 0.0000  & 0.6230  & 0.7500  & 0.8864  & 0.9313  & 0.9440  \\ \hline
$\|Q_h u_5-u_{5,h}\|$ & 7.4542e-1 & 1.1915e-1 & 2.9679e-2 & 7.5124e-3 & 1.9132e-3 & 4.8882e-4 \\ \hline
order & 0.0000  & 2.6453  & 2.0053  & 1.9821  & 1.9733  & 1.9686  \\ \hline
$\trb{Q_h u_6-u_{6,h}}$ & 6.9213e+0 & 4.9834e+0 & 2.9631e+0 & 1.6029e+0 & 8.4057e-1 & 4.3693e-1 \\ \hline
order & 0.0000  & 0.4739  & 0.7500  & 0.8864  & 0.9313  & 0.9440  \\ \hline
$\|Q_h u_6-u_{6,h}\|$ & 4.8677e-1 & 1.1915e-1 & 2.9679e-2 & 7.5124e-3 & 1.9132e-3 & 4.8882e-4 \\ \hline
order & 0.0000  & 2.0304  & 2.0053  & 1.9821  & 1.9733  & 1.9686  \\ \hline
\end{tabular}
\end{table}
}

\begin{table}[]
\centering \caption{Convergence rates for $\eps=0.05$ and $k=1$.}\label{Tables_Exam1_8}
\begin{tabular}{|c|c|c|c|c|c|c|c|c|c|c|c|c|}
\hline
$h$     &    1/4   &    1/8   &    1/16  &    1/32  &    1/64  &    1/128 \\ \hline
$\trb{Q_h u_1-u_{1,h}}$ & 2.0408e+0 & 1.0776e+0 & 5.5472e-1 & 2.8361e-1 & 1.4474e-1 & 7.3831e-2 \\ \hline
order &       & 0.9213  & 0.9580  & 0.9678  & 0.9704  & 0.9712  \\ \hline
$\|Q_h u_1-u_{1,h}\|$ & 2.2929e-1 & 6.1245e-2 & 1.6035e-2 & 4.1758e-3 & 1.0858e-3 & 2.8224e-4 \\ \hline
order &       & 1.9045  & 1.9333  & 1.9411  & 1.9432  & 1.9438  \\ \hline
$\trb{Q_h u_2-u_{2,h}}$ & 4.0525e+0 & 2.5606e+0 & 1.4016e+0 & 7.2851e-1 & 3.7309e-1 & 1.9031e-1 \\ \hline
order &       & 0.6623  & 0.8693  & 0.9441  & 0.9654  & 0.9711  \\ \hline
$\|Q_h u_2-u_{2,h}\|$ & 2.0213e-1 & 4.9226e-2 & 1.2287e-2 & 3.0925e-3 & 7.8035e-4 & 1.9713e-4 \\ \hline
order &       & 2.0378  & 2.0023  & 1.9903  & 1.9866  & 1.9850  \\ \hline
$\trb{Q_h u_3-u_{3,h}}$ & 4.0525e+0 & 2.5606e+0 & 1.4016e+0 & 7.2851e-1 & 3.7309e-1 & 1.9031e-1 \\ \hline
order &       & 0.6623  & 0.8693  & 0.9441  & 0.9654  & 0.9711  \\ \hline
$\|Q_h u_3-u_{3,h}\|$ & 2.0213e-1 & 4.9226e-2 & 1.2287e-2 & 3.0925e-3 & 7.8035e-4 & 1.9713e-4 \\ \hline
order &       & 2.0378  & 2.0023  & 1.9903  & 1.9866  & 1.9850  \\ \hline
$\trb{Q_h u_4-u_{4,h}}$ & 7.4014e+0 & 4.1637e+0 & 2.1988e+0 & 1.1318e+0 & 5.7863e-1 & 2.9528e-1 \\ \hline
order &       & 0.8299  & 0.9211  & 0.9581  & 0.9679  & 0.9705  \\ \hline
$\|Q_h u_4-u_{4,h}\|$ & 8.7222e-1 & 2.3834e-1 & 6.3677e-2 & 1.6672e-2 & 4.3413e-3 & 1.1288e-3 \\ \hline
order &       & 1.8717  & 1.9042  & 1.9333  & 1.9412  & 1.9433  \\ \hline
$\trb{Q_h u_5-u_{5,h}}$ & 6.8886e+0 & 4.8293e+0 & 2.8003e+0 & 1.4825e+0 & 7.6263e-1 & 3.8924e-1 \\ \hline
order &       & 0.5124  & 0.7862  & 0.9176  & 0.9590  & 0.9703  \\ \hline
$\|Q_h u_5-u_{5,h}\|$ & 4.9173e-1 & 1.1747e-1 & 2.8947e-2 & 7.2442e-3 & 1.8217e-3 & 4.5881e-4 \\ \hline
order &       & 2.0656  & 2.0208  & 1.9985  & 1.9916  & 1.9893  \\ \hline
$\trb{Q_h u_6-u_{6,h}}$ & 6.8882e+0 & 4.8293e+0 & 2.8003e+0 & 1.4825e+0 & 7.6263e-1 & 3.8924e-1 \\ \hline
order &       & 0.5123  & 0.7862  & 0.9176  & 0.9590  & 0.9703  \\ \hline
$\|Q_h u_6-u_{6,h}\|$ & 4.9159e-1 & 1.1747e-1 & 2.8947e-2 & 7.2442e-3 & 1.8217e-3 & 4.5881e-4 \\ \hline
order &       & 2.0652  & 2.0208  & 1.9985  & 1.9916  & 1.9893  \\ \hline
\end{tabular}
\end{table}

\begin{table}[]
\centering \caption{Convergence rates for $\eps=0$ and $k=1$.}\label{Tables_Exam1_9}
\begin{tabular}{|c|c|c|c|c|c|c|c|c|c|c|c|c|}
\hline
$h$    &    1/4   &    1/8   &    1/16  &    1/32  &    1/64  &    1/128 \\ \hline \hline
$\trb{Q_h u_1-u_{1,h}}$ & 1.9806e+0 & 1.0247e+0 & 5.1693e-1 & 2.5905e-1 & 1.2960e-1 & 6.4808e-2 \\ \hline
order &       & 0.9508  & 0.9871  & 0.9967  & 0.9992  & 0.9998  \\ \hline
$\|Q_h u_1-u_{1,h}\|$ & 2.1635e-1 & 5.5551e-2 & 1.3986e-2 & 3.5027e-3 & 8.7607e-4 & 2.1904e-4 \\ \hline
order &       & 1.9615  & 1.9898  & 1.9974  & 1.9994  & 1.9998  \\ \hline
$\trb{Q_h u_2-u_{2,h}}$ & 3.9877e+0 & 2.4601e+0 & 1.3162e+0 & 6.7019e-1 & 3.3666e-1 & 1.6852e-1 \\ \hline
order &       & 0.6968  & 0.9024  & 0.9737  & 0.9933  & 0.9983  \\ \hline
$\|Q_h u_2-u_{2,h}\|$ & 2.0039e-1 & 4.8309e-2 & 1.1944e-2 & 2.9773e-3 & 7.4379e-4 & 1.8591e-4 \\ \hline
order &       & 2.0524  & 2.0160  & 2.0042  & 2.0011  & 2.0003  \\ \hline
$\trb{Q_h u_3-u_{3,h}}$ & 3.9877e+0 & 2.4601e+0 & 1.3162e+0 & 6.7019e-1 & 3.3666e-1 & 1.6852e-1 \\ \hline
order &       & 0.6968  & 0.9024  & 0.9737  & 0.9933  & 0.9983  \\ \hline
$\|Q_h u_3-u_{3,h}\|$ & 2.0039e-1 & 4.8309e-2 & 1.1944e-2 & 2.9773e-3 & 7.4379e-4 & 1.8591e-4 \\ \hline
order &       & 2.0524  & 2.0160  & 2.0042  & 2.0011  & 2.0003  \\ \hline
$\trb{Q_h u_4-u_{4,h}}$ & 1.2992e+1 & 3.9612e+0 & 2.0494e+0 & 1.0339e+0 & 5.1810e-1 & 2.5920e-1 \\ \hline
order &       & 1.7136  & 0.9508  & 0.9871  & 0.9967  & 0.9992  \\ \hline
$\|Q_h u_4-u_{4,h}\|$ & 1.9731e+0 & 2.1635e-1 & 5.5551e-2 & 1.3986e-2 & 3.5027e-3 & 8.7607e-4 \\ \hline
order &       & 3.1890  & 1.9615  & 1.9898  & 1.9974  & 1.9994  \\ \hline
$\trb{Q_h u_5-u_{5,h}}$ & 6.8652e+0 & 4.6781e+0 & 2.6456e+0 & 1.3710e+0 & 6.9191e-1 & 3.4677e-1 \\ \hline
order &       & 0.5534  & 0.8223  & 0.9484  & 0.9865  & 0.9966  \\ \hline
$\|Q_h u_5-u_{5,h}\|$ & 4.9983e-1 & 1.1603e-1 & 2.8373e-2 & 7.0491e-3 & 1.7595e-3 & 4.3969e-4 \\ \hline
order &       & 2.1069  & 2.0319  & 2.0090  & 2.0023  & 2.0006  \\ \hline
$\trb{Q_h u_6-u_{6,h}}$ & 8.5941e+0 & 4.6781e+0 & 2.6456e+0 & 1.3710e+0 & 6.9191e-1 & 3.4677e-1 \\ \hline
order &       & 0.8774  & 0.8223  & 0.9484  & 0.9865  & 0.9966  \\ \hline
$\|Q_h u_6-u_{6,h}\|$ & 9.9163e-1 & 1.1603e-1 & 2.8373e-2 & 7.0491e-3 & 1.7595e-3 & 4.3969e-4 \\ \hline
order &       & 3.0953  & 2.0319  & 2.0090  & 2.0023  & 2.0006  \\ \hline
\end{tabular}
\end{table}

\begin{table}[]
\centering \caption{Convergence rates for $\eps=0.1$ and $k=2$.}\label{Tables_Exam1_10}
\begin{tabular}{|c|c|c|c|c|c|c|c|c|c|c|c|c|}
\hline
$h$    &    1/4   &    1/8   &    1/16  &    1/32  &    1/64  &    1/128 \\ \hline \hline
$\trb{Q_h u_1-u_{1,h}}$ & 4.1653e-1 & 1.0911e-1 & 2.8410e-2 & 7.3891e-3 & 1.9213e-3 & 4.9944e-4 \\ \hline
order &       & 1.9327  & 1.9413  & 1.9429  & 1.9433  & 1.9437  \\ \hline
$\|Q_h u_1-u_{1,h}\|$ & 8.5534e-2 & 1.0676e-2 & 1.4001e-3 & 1.8702e-4 & 2.5108e-5 & 3.3748e-6 \\ \hline
order &       & 3.0021  & 2.9308  & 2.9043  & 2.8970  & 2.8953  \\ \hline
$\trb{Q_h u_2-u_{2,h}}$ & 1.5575e+0 & 4.3076e-1 & 1.1323e-1 & 2.9496e-2 & 7.6668e-3 & 1.9916e-3 \\ \hline
order &       & 1.8543  & 1.9277  & 1.9406  & 1.9438  & 1.9447  \\ \hline
$\|Q_h u_2-u_{2,h}\|$ & 3.2591e-1 & 4.1406e-2 & 5.2583e-3 & 6.9630e-4 & 9.3495e-5 & 1.2600e-5 \\ \hline
order &       & 2.9766  & 2.9772  & 2.9168  & 2.8967  & 2.8914  \\ \hline
$\trb{Q_h u_3-u_{3,h}}$ & 1.7381e+0 & 4.4628e-1 & 1.1430e-1 & 2.9567e-2 & 7.6714e-3 & 1.9919e-3 \\ \hline
order &       & 1.9615  & 1.9651  & 1.9508  & 1.9464  & 1.9453  \\ \hline
$\|Q_h u_3-u_{3,h}\|$ & 3.5920e-1 & 4.5291e-2 & 5.7680e-3 & 7.6165e-4 & 1.0170e-4 & 1.3623e-5 \\ \hline
order &       & 2.9875  & 2.9731  & 2.9209  & 2.9047  & 2.9003  \\ \hline
$\trb{Q_h u_4-u_{4,h}}$ & 3.1633e+0 & 8.8821e-1 & 2.2815e-1 & 5.9192e-2 & 1.5378e-2 & 3.9962e-3 \\ \hline
order &       & 1.8324  & 1.9609  & 1.9465  & 1.9446  & 1.9441  \\ \hline
$\|Q_h u_4-u_{4,h}\|$ & 6.6471e-1 & 9.6376e-2 & 1.1636e-2 & 1.5123e-3 & 2.0146e-4 & 2.7020e-5 \\ \hline
order &       & 2.7860  & 3.0501  & 2.9438  & 2.9082  & 2.8984  \\ \hline
$\trb{Q_h u_5-u_{5,h}}$ & 3.3748e+0 & 1.2150e+0 & 3.2184e-1 & 8.4035e-2 & 2.1843e-2 & 5.6714e-3 \\ \hline
order &       & 1.4738  & 1.9166  & 1.9373  & 1.9438  & 1.9454  \\ \hline
$\|Q_h u_5-u_{5,h}\|$ & 6.5236e-1 & 1.1991e-1 & 1.4765e-2 & 1.9227e-3 & 2.5710e-4 & 3.4647e-5 \\ \hline
order &       & 2.4437  & 3.0217  & 2.9410  & 2.9027  & 2.8915  \\ \hline
$\trb{Q_h u_6-u_{6,h}}$ & 4.2203e+0 & 1.3252e+0 & 3.2987e-1 & 8.4571e-2 & 2.1878e-2 & 5.6737e-3 \\ \hline
order &       & 1.6711  & 2.0063  & 1.9637  & 1.9507  & 1.9471  \\ \hline
$\|Q_h u_6-u_{6,h}\|$ & 7.2192e-1 & 1.3575e-1 & 1.6679e-2 & 2.1691e-3 & 2.8817e-4 & 3.8522e-5 \\ \hline
order &       & 2.4109  & 3.0248  & 2.9429  & 2.9121  & 2.9032  \\ \hline
\end{tabular}
\end{table}

\begin{table}[]
\centering \caption{Convergence rates for $\eps=0.05$ and $k=2$.}\label{Tables_Exam1_11}
\begin{tabular}{|c|c|c|c|c|c|c|c|c|c|c|c|c|}
\hline
$h$      &    1/4   &    1/8   &    1/16  &    1/32  &    1/64  &    1/128 \\ \hline \hline
$\trb{Q_h u_1-u_{1,h}}$ & 4.0028e-1 & 1.0330e-1 & 2.6433e-2 & 6.7483e-3 & 1.7218e-3 & 4.3922e-4 \\ \hline
order &       & 1.9542  & 1.9664  & 1.9698  & 1.9706  & 1.9709  \\ \hline
$\|Q_h u_1-u_{1,h}\|$ & 7.9910e-2 & 9.6767e-3 & 1.2274e-3 & 1.5829e-4 & 2.0503e-5 & 2.6586e-6 \\ \hline
order &       & 3.0458  & 2.9789  & 2.9550  & 2.9486  & 2.9471  \\ \hline
$\trb{Q_h u_2-u_{2,h}}$ & 1.4826e+0 & 4.0594e-1 & 1.0531e-1 & 2.6977e-2 & 6.8859e-3 & 1.7561e-3 \\ \hline
order &       & 1.8688  & 1.9466  & 1.9649  & 1.9700  & 1.9713  \\ \hline
$\|Q_h u_2-u_{2,h}\|$ & 2.9953e-1 & 3.7029e-2 & 4.5641e-3 & 5.8360e-4 & 7.5522e-5 & 9.8044e-6 \\ \hline
order &       & 3.0160  & 3.0203  & 2.9673  & 2.9500  & 2.9454  \\ \hline
$\trb{Q_h u_3-u_{3,h}}$ & 1.6545e+0 & 4.2057e-1 & 1.0631e-1 & 2.7041e-2 & 6.8900e-3 & 1.7563e-3 \\ \hline
order &       & 1.9760  & 1.9840  & 1.9751  & 1.9726  & 1.9719  \\ \hline
$\|Q_h u_3-u_{3,h}\|$ & 3.3275e-1 & 4.0902e-2 & 5.0762e-3 & 6.4953e-4 & 8.3845e-5 & 1.0847e-5 \\ \hline
order &       & 3.0242  & 3.0104  & 2.9663  & 2.9536  & 2.9504  \\ \hline
$\trb{Q_h u_4-u_{4,h}}$ & 3.0151e+0 & 8.3014e-1 & 2.1137e-1 & 5.3984e-2 & 1.3775e-2 & 3.5138e-3 \\ \hline
order &       & 1.8608  & 1.9736  & 1.9692  & 1.9705  & 1.9709  \\ \hline
$\|Q_h u_4-u_{4,h}\|$ & 6.2435e-1 & 8.5941e-2 & 1.0140e-2 & 1.2774e-3 & 1.6440e-4 & 2.1281e-5 \\ \hline
order &       & 2.8609  & 3.0833  & 2.9888  & 2.9579  & 2.9496  \\ \hline
$\trb{Q_h u_5-u_{5,h}}$ & 3.2512e+0 & 1.1345e+0 & 2.9867e-1 & 7.6888e-2 & 1.9646e-2 & 5.0100e-3 \\ \hline
order &       & 1.5189  & 1.9255  & 1.9577  & 1.9685  & 1.9713  \\ \hline
$\|Q_h u_5-u_{5,h}\|$ & 6.0052e-1 & 1.0549e-1 & 1.2715e-2 & 1.6035e-3 & 2.0673e-4 & 2.6825e-5 \\ \hline
order &       & 2.5090  & 3.0525  & 2.9873  & 2.9554  & 2.9461  \\ \hline
$\trb{Q_h u_6-u_{6,h}}$ & 4.0657e+0 & 1.2374e+0 & 3.0612e-1 & 7.7378e-2 & 1.9678e-2 & 5.0120e-3 \\ \hline
order &       & 1.7162  & 2.0152  & 1.9841  & 1.9754  & 1.9731  \\ \hline
$\|Q_h u_6-u_{6,h}\|$ & 6.8404e-1 & 1.2118e-1 & 1.4639e-2 & 1.8523e-3 & 2.3824e-4 & 3.0778e-5 \\ \hline
order &       & 2.4969  & 3.0493  & 2.9824  & 2.9588  & 2.9525  \\ \hline
\end{tabular}
\end{table}

\begin{table}[]
\centering \caption{Convergence rates for $\eps=0$ and $k=2$.}\label{Tables_Exam1_12}
\begin{tabular}{|c|c|c|c|c|c|c|c|c|c|c|c|c|}
\hline
$h$    &    1/4   &    1/8   &    1/16  &    1/32  &    1/64  &    1/128 \\ \hline \hline
$\trb{Q_h u_1-u_{1,h}}$ & 3.8476e-1 & 9.7751e-2 & 2.4561e-2 & 6.1490e-3 & 1.5379e-3 & 3.8452e-4 \\ \hline
order &       & 1.9768  & 1.9927  & 1.9980  & 1.9994  & 1.9998  \\ \hline
$\|Q_h u_1-u_{1,h}\|$ & 7.4702e-2 & 8.7708e-3 & 1.0752e-3 & 1.3375e-4 & 1.6701e-5 & 2.0872e-6 \\ \hline
order &       & 3.0904  & 3.0281  & 3.0070  & 3.0015  & 3.0003  \\ \hline
$\trb{Q_h u_2-u_{2,h}}$ & 1.4123e+0 & 3.8267e-1 & 9.7879e-2 & 2.4631e-2 & 6.1684e-3 & 1.5428e-3 \\ \hline
order &       & 1.8839  & 1.9670  & 1.9906  & 1.9975  & 1.9993  \\ \hline
$\|Q_h u_2-u_{2,h}\|$ & 2.7564e-1 & 3.3148e-2 & 3.9609e-3 & 4.8855e-4 & 6.0870e-5 & 7.6034e-6 \\ \hline
order &       & 3.0558  & 3.0650  & 3.0192  & 3.0047  & 3.0010  \\ \hline
$\trb{Q_h u_3-u_{3,h}}$ & 1.5761e+0 & 3.9645e-1 & 9.8807e-2 & 2.4690e-2 & 6.1721e-3 & 1.5431e-3 \\ \hline
order &       & 1.9911  & 2.0045  & 2.0007  & 2.0001  & 2.0000  \\ \hline
$\|Q_h u_3-u_{3,h}\|$ & 3.0857e-1 & 3.6994e-2 & 4.4723e-3 & 5.5453e-4 & 6.9212e-5 & 8.6507e-6 \\ \hline
order &       & 3.0603  & 3.0482  & 3.0117  & 3.0022  & 3.0001  \\ \hline
$\trb{Q_h u_4-u_{4,h}}$ & 2.8709e+0 & 7.7652e-1 & 1.9567e-1 & 4.9133e-2 & 1.2299e-2 & 3.0759e-3 \\ \hline
order &       & 1.8864  & 1.9886  & 1.9937  & 1.9981  & 1.9995  \\ \hline
$\|Q_h u_4-u_{4,h}\|$ & 5.8517e-1 & 7.6776e-2 & 8.8369e-3 & 1.0776e-3 & 1.3385e-4 & 1.6705e-5 \\ \hline
order &       & 2.9301  & 3.1190  & 3.0358  & 3.0091  & 3.0022  \\ \hline
$\trb{Q_h u_5-u_{5,h}}$ & 3.1281e+0 & 1.0609e+0 & 2.7714e-1 & 7.0267e-2 & 1.7634e-2 & 4.4131e-3 \\ \hline
order &       & 1.5600  & 1.9366  & 1.9797  & 1.9945  & 1.9985  \\ \hline
$\|Q_h u_5-u_{5,h}\|$ & 5.5149e-1 & 9.3075e-2 & 1.0962e-2 & 1.3372e-3 & 1.6608e-4 & 2.0730e-5 \\ \hline
order &       & 2.5669  & 3.0858  & 3.0352  & 3.0093  & 3.0021  \\ \hline
$\trb{Q_h u_6-u_{6,h}}$ & 3.9117e+0 & 1.1571e+0 & 2.8405e-1 & 7.0715e-2 & 1.7663e-2 & 4.4148e-3 \\ \hline
order &       & 1.7573  & 2.0263  & 2.0061  & 2.0013  & 2.0003  \\ \hline
$\|Q_h u_6-u_{6,h}\|$ & 6.4688e-1 & 1.0852e-1 & 1.2881e-2 & 1.5861e-3 & 1.9765e-4 & 2.4697e-5 \\ \hline
order &       & 2.5756  & 3.0746  & 3.0217  & 3.0044  & 3.0005  \\ \hline
\end{tabular}
\end{table}

\subsection{L shape domain}
Now we consider the eigenvalue problem
(\ref{problem-eq}) on the L shape domain
$\Omega=(-1,1)^2\backslash (0,1)^2$.

We also use the weak Galerkin method to solve this eigenvalue problem and Table \ref{Tables_Exam2_1}-\ref{Tables_Exam2_6}
presents the corresponding numerical results for the first six eigenvalues. Even the analytic eigenpairs is not known,
 from these tables, we can find the numerical
eigenvalues $\lambda_{j,h}$ increases when $h$ decreases which shows that $\lambda_{j,h}$ is a lower
bound of the exact eigenvalue $\lambda_j$.

\begin{table}[]
\centering \caption{Discrete eigenvalues for $\eps=0.1$ and $k=1$.}\label{Tables_Exam2_1}
\begin{tabular}{|c|c|c|c|c|c|c|c|c|}
\hline
$h$    &    1/4   &    1/8   &    1/16  &    1/32  &    1/64  &    1/128 \\ \hline \hline
$\lambda_{1,h}$   & 8.0444  & 9.1197  & 9.4787  & 9.5893  & 9.6234  & 9.6343  \\ \hline
$\lambda_{2,h}$   & 12.1745  & 14.2566  & 14.9345  & 15.1263  & 15.1782  & 15.1922  \\ \hline
$\lambda_{3,h}$   & 15.0478  & 18.2342  & 19.3145  & 19.6240  & 19.7083  & 19.7309  \\ \hline
$\lambda_{4,h}$   & 19.7958  & 26.1541  & 28.5461  & 29.2554  & 29.4501  & 29.5024  \\ \hline
$\lambda_{5,h}$   & 20.2283  & 27.6970  & 30.6403  & 31.5476  & 31.8078  & 31.8819  \\ \hline
$\lambda_{6,h}$   & 23.7850  & 34.6725  & 39.3995  & 40.8900  & 41.3124  & 41.4292  \\ \hline
\end{tabular}
\end{table}

\begin{table}[]
\centering \caption{Discrete eigenvalues for $\eps=0.05$ and $k=1$.}\label{Tables_Exam2_2}
\begin{tabular}{|c|c|c|c|c|c|c|c|c|}
\hline
$h$     &    1/4   &    1/8   &    1/16  &    1/32  &    1/64  &    1/128 \\ \hline \hline
$\lambda_{1,h}$   & 8.0934  & 9.1472  & 9.4896  & 9.5931  & 9.6247  & 9.6346  \\ \hline
$\lambda_{2,h}$   & 12.2877  & 14.3238  & 14.9615  & 15.1356  & 15.1813  & 15.1931  \\ \hline
$\lambda_{3,h}$   & 15.2218  & 18.3444  & 19.3598  & 19.6398  & 19.7134  & 19.7325  \\ \hline
$\lambda_{4,h}$   & 20.1002  & 26.3819  & 28.6452  & 29.2905  & 29.4615  & 29.5060  \\ \hline
$\lambda_{5,h}$   & 20.5464  & 27.9527  & 30.7545  & 31.5885  & 31.8212  & 31.8861  \\ \hline
$\lambda_{6,h}$   & 24.2289  & 35.0750  & 39.5886  & 40.9586  & 41.3349  & 41.4363  \\ \hline
\end{tabular}
\end{table}

\begin{table}[]
\centering \caption{Discrete eigenvalues for $\eps=0$ and $k=1$.}\label{Tables_Exam2_3}
\begin{tabular}{|c|c|c|c|c|c|c|c|c|}
\hline
$h$   &    1/4   &    1/8   &    1/16  &    1/32  &    1/64  &    1/128 \\ \hline \hline
$\lambda_{1,h}$   & 8.1405  & 9.1725  & 9.4992  & 9.5963  & 9.6257  & 9.6349  \\ \hline
$\lambda_{2,h}$   & 12.3971  & 14.3860  & 14.9856  & 15.1437  & 15.1838  & 15.1939  \\ \hline
$\lambda_{3,h}$   & 15.3906  & 18.4466  & 19.4000  & 19.6534  & 19.7177  & 19.7338  \\ \hline
$\lambda_{4,h}$   & 20.3977  & 26.5942  & 28.7336  & 29.3206  & 29.4710  & 29.5088  \\ \hline
$\lambda_{5,h}$   & 20.8576  & 28.1912  & 30.8564  & 31.6235  & 31.8322  & 31.8894  \\ \hline
$\lambda_{6,h}$   & 24.6655  & 35.4522  & 39.7576  & 41.0175  & 41.3535  & 41.4419  \\ \hline
\end{tabular}
\end{table}

\begin{table}[]
\centering \caption{Discrete eigenvalues for $\eps=0.1$ and $k=2$.}\label{Tables_Exam2_4}
\begin{tabular}{|c|c|c|c|c|c|c|c|c|}
\hline
$h$     &    1/4   &    1/8   &    1/16  &    1/32  &    1/64  &    1/128 \\ \hline \hline
$\lambda_{1,h}$   & 9.5538  & 9.6152  & 9.6306  & 9.6362  & 9.6383  & 9.6392  \\ \hline
$\lambda_{2,h}$   & 15.0957  & 15.1903  & 15.1967  & 15.1972  & 15.1972  & 15.1973  \\ \hline
$\lambda_{3,h}$   & 19.5148  & 19.7251  & 19.7383  & 19.7391  & 19.7392  & 19.7392  \\ \hline
$\lambda_{4,h}$   & 28.7653  & 29.4755  & 29.5185  & 29.5213  & 29.5215  & 29.5215  \\ \hline
$\lambda_{5,h}$   & 30.6553  & 31.7870  & 31.8860  & 31.9036  & 31.9091  & 31.9113  \\ \hline
$\lambda_{6,h}$   & 39.0485  & 41.2933  & 41.4490  & 41.4674  & 41.4719  & 41.4735  \\ \hline
\end{tabular}
\end{table}

\begin{table}[]
\centering \caption{Discrete eigenvalues for $\eps=0.05$ and $k=2$.}\label{Tables_Exam2_5}
\begin{tabular}{|c|c|c|c|c|c|c|c|c|}
\hline
$h$   &    1/4   &    1/8   &    1/16  &    1/32  &    1/64  &    1/128 \\ \hline \hline
$\lambda_{1,h}$   & 9.5555  & 9.6155  & 9.6307  & 9.6362  & 9.6383  & 9.6392  \\ \hline
$\lambda_{2,h}$   & 15.1012  & 15.1908  & 15.1968  & 15.1972  & 15.1972  & 15.1973  \\ \hline
$\lambda_{3,h}$   & 19.5283  & 19.7264  & 19.7384  & 19.7392  & 19.7392  & 19.7392  \\ \hline
$\lambda_{4,h}$   & 28.8132  & 29.4796  & 29.5188  & 29.5213  & 29.5215  & 29.5215  \\ \hline
$\lambda_{5,h}$   & 30.7260  & 31.7934  & 31.8867  & 31.9037  & 31.9092  & 31.9113  \\ \hline
$\lambda_{6,h}$   & 39.2028  & 41.3060  & 41.4501  & 41.4675  & 41.4719  & 41.4735  \\ \hline
\end{tabular}
\end{table}

\begin{table}[]
\centering \caption{Discrete eigenvalues for $\eps=0$ and $k=2$.}\label{Tables_Exam2_6}
\begin{tabular}{|c|c|c|c|c|c|c|c|c|}
\hline
$h$     &    1/4   &    1/8   &    1/16  &    1/32  &    1/64  &    1/128 \\ \hline \hline
$\lambda_{1,h}$   & 9.5571  & 9.6157  & 9.6308  & 9.6362  & 9.6383  & 9.6392  \\ \hline
$\lambda_{2,h}$   & 15.1064  & 15.1913  & 15.1968  & 15.1972  & 15.1972  & 15.1973  \\ \hline
$\lambda_{3,h}$   & 19.5410  & 19.7275  & 19.7385  & 19.7392  & 19.7392  & 19.7392  \\ \hline
$\lambda_{4,h}$   & 28.8577  & 29.4833  & 29.5191  & 29.5213  & 29.5215  & 29.5215  \\ \hline
$\lambda_{5,h}$   & 30.7916  & 31.7993  & 31.8873  & 31.9038  & 31.9092  & 31.9113  \\ \hline
$\lambda_{6,h}$   & 39.3450  & 41.3176  & 41.4512  & 41.4676  & 41.4719  & 41.4735  \\ \hline
\end{tabular}
\end{table}

\section{Concluding remarks}
In this paper, we apply the weak Galerkin method to solve the
eigenvalue problems and the corresponding convergence analysis is
also given. Furthermore, we also analyze the lower-bound property of
the weak Galerkin method. {\color{blue}Compared with the classical
nonconforming finite element method which can provide lower bound
approximation by linear element with only the second order
convergence, the weak Galerkin method can provide lower bound
approximation with a high order convergence (larger than $2$)}.

In the future, it is required to design the efficient solver for the algebraic eigenvalue problems
derived by the weak Galerkin method.

\end{document}